\documentclass{article}
\usepackage{amsmath,amssymb,latexsym,amsthm,color}
\usepackage{hyperref}
\definecolor{amethyst}{rgb}{0.6, 0.4, 0.8}
\definecolor{blue-violet}{rgb}{0.54, 0.17, 0.89}

\newcommand{\KK}{\mathbb{K}}
\newcommand{\QQ}{\mathbb{Q}}

\newcommand{\NN}{\mathbb{N}}
\newcommand{\ZZ}{\mathbb{Z}}
\newcommand{\LL}{\mathbb{L}}

\newcommand{\coker}{\mathrm{coker}\kern.5pt}

\theoremstyle{plain}
\newtheorem{theorem}{Theorem}[section]
\newtheorem{corollary}[theorem]{Corollary}
\newtheorem{proposition}[theorem]{Proposition}
\newtheorem{lemma}[theorem]{Lemma}

\theoremstyle{definition}
\newtheorem{definition}[theorem]{Definition}
\newtheorem{example}[theorem]{Example}

\newtheorem{remark}[theorem]{Remark}

\begin{document}

\title{On the Intermediate Value Theorem over a Valued Field}

\author{Carla Massaza  Lea Terracini  Paolo Valabrega  }
 
\date{}

\maketitle

\noindent {\bf Abstract: } The paper proves the intermediate value theorem for polynomials and power series over a valued field with divisible valuation group and infinite residue field. Some further results on the behaviour of the valuation are obtained using Hensel's Lemma.\\


\section{Introduction}

The intermediate value theorem (IVT for short) for a continuous function over the field of real numbers is a well-known, long ago established result of Mathematical Analysis. 

It is also well-known that it does not hold over a non-Archimedean ordered field $\KK$, even if it is maximal ordered and complete (see \cite{Corgnier}). However Bourbaki (see \cite{Bourbaki II}) shows that IVT holds true for any polynomial over a maximal ordered non-Archimedean field. As for power series, in \cite{Corgnier} and \cite{CorgnierII}  it is proved that IVT holds true over a complete maximal ordered non-Archimedean field.

When $\KK$ is a field equipped with a valuation $v: \KK \to G \cup \{\infty\}, G$ being an ordered group, IVT can be considered but it makes sense if we give it a different meaning.

On the field of real numbers, as well as on a non-Archimedean ordered field, IVT for a set $F$ of functions (polynomials, power series, $\ldots$) states that, given any $f \in F$ and any closed interval $[a,b] \subset \KK$,  if $f(a)f(b) < 0$, then there is $c \in ]a,b[$ such that $f(c) = 0$.

On a valued field $\KK$ condition $f(a)f(b) < 0$ makes no sense, so we replace it by condition $v(f(a)) > 0, v(f(b)) < 0$ (or conversely); as for condition $c \in ]a,b[ \subset \KK$, we will see that $ v(c) \in [v(a),v(b)] \subset G$ (closed interval) is the right replacement. We want to point out that there are cases that force $v(c)$ to coincide with one endpoint (and the endpoints might coincide).

The aim of the present paper is the investigation of an analogue of IVT for polynomials and power series  over a valued field (see \cite {EP}, \cite {Ribenboim1}), where analogue means that the above conditions are fulfilled. It is worth observing that, over a  non-Archimedean ordered field, the set of those $c$ that satisfy IVT for a given polynomial or power series  has finite cardinality (one in most cases), while in general on a valued field such set has infinite cardinality.

The main effort of this paper consists of a proof of IVT for polynomials. The result  holds true for polynomials under the following two assumptions:
\begin{equation}
\label{eq:Hypo}
\begin{array} {l}
$(i)$ \hbox{ the valuation group is divisible,} \\
$(ii)$ \hbox{ the residue field of the valuation is an infinite set.}
\end{array}
\end{equation}

 Actually, we prove more than IVT for polynomials over the valued field $\KK$; in fact we consider polynomials $f(X)$ defined over any overfield $\LL \supseteq \KK$ with a valuation extending the valuation of $\KK$ (with the same ordered group)  and investigate the function $x \to f(x), \ x \in \KK, \ f(x)$ being allowed to lie outside of $\KK$. 

It is easy to see that, from IVT, under our assumptions, the following extended statement can be obtained:
 
 if $v(f(b)) < \alpha < v(f(a))$ (or conversely), then there is $c \in \KK$ such that $v(f(c)) = \alpha$ and $v(c)$ lies between $v(a),v(b)$, not excluding the endpoints. 

We also show that our two assumptions (divisibility and infinite cardinality of the residue field) cannot be avoided.

The extension to power series (that requires the above extended statement) can be obtained (over a complete valued field) by using the property that IVT passes from a uniformly converging sequence of functions to its limit.

Alternatively the result can be attained by using Hensel's Lemma for the ring of restricted power series over the valuation ring of $v$,   with the same technique used in \cite{Valabrega}. 
Hensel's Lemma, by decomposing a series as a product of a polynomial and a series that reduces to $1$ in the residue field,  gives further side results on the valuation behaviour of  power series (Theorem \ref{teo:seriesbehaviour}) .

 We wish to point out that neither $\KK$ nor the overfield $\LL$ are assumed to be complete, or henselian or algebraically closed when IVT for polynomials is proved. When dealing with power series our techniques require an embedding of $\KK$ into a complete henselian extension (that is always existing under our hypotheses). 
 
 We also wish to notice that, in the case of $p$-adic valuations over a complete algebraically closed field, IVT, as stated in this paper for polynomials and power series, is also proved in \cite[section 2.5, p. 317]{Robert},  where well-established  techniques of $p$-adic analysis are used.

In the present  paper we use quite different techniques, both for polynomials and for power series and obtain results including fields that are not algebraically closed. Moreover, despite the countability assumption on the valuation topology, our results hold true for a wide class of valuations far away from the $p$-adic one, for instance in the case of a non-Archimedean ordered group of values whose topology is countable.

\section{The Intermediate Value Theorem for polynomials}\label{sect:due}

\subsection{General facts}
$\KK$ is a field equipped with a non-trivial valuation, i.e. an onto map $v: \KK \to G \cup \{\infty\}$, where $G$ is an ordered group, satisfying some additional properties (see \cite{Bourbaki}, \cite{Bourbaki II}, \cite{Ribenboim1}, \cite{Ribenboim2}). The valuation gives rise to the valuation ring $(A,M)$, where $a \in A$ if and only if $v(a) \ge 0$, while $m \in M$ if and only if $v(m) > 0$. It gives also rise to a topology $\tau$. Hence $\KK$ becomes a topological field and $A$, equipped with the restricted topology, is easily seen to be a Hausdorff topological ring (see \cite{EP}, p. 45). Such a topology is also easily seen to be linear.  The completion $\hat \KK$ of $\KK$ is a valued field and $v$ can be uniquely extended to $\hat \KK$ (\cite{Ribenboim1}).

If $\bar \KK$ is an algebraic closure of $\KK$, then $v$ can be extended, usually in many ways, to a valuation $\bar v$ of $\bar \KK$ whose group is the divisible group $G'$  generated by $G$ (\cite{Ribenboim1}). The extension is unique if and only if $\KK$ is henselian (\cite{Ribenboim1}). We observe that, if $G$ is divisible, then $G= G'$.

If $\LL \supset \KK$ is any overfield, then $v$ can be extended (not uniquely in general) to a valuation $v_L$ of the field $\LL$ with ordered group $G'$ containing $G$ as a subgroup (see \cite{Ribenboim1}, p.45, Theorem 5);  we will focus our attention on the case $G = G'$. It is known that, if $G$ is divisible, a valuation $v_L$ with $G' = G$ always exists.

In what follows we consider an extension $\LL \supset \KK$,  equipped with a fixed extension $v_L$ of $v$, and, for the sake of simplicity, since there is no ambiguity, we use the symbol $v$ also for $v_L$. In particular we use $v$ for a fixed extension of $v$ to the algebraic closure $\bar \KK$. Observe that, if $f(x)$ is any function defined on $\LL$ and taking its values in $\LL$, then it is well-defined $f_K: \KK \to \LL$, restriction of $f(x)$ to $\KK$. 
 \medskip

\noindent NOTATION. Once for all the symbol $]\{x,y\}[$  (respectively $\ [\{x,y\}], ]\{x,y\}],[\{x,y\}[)$   denotes the interval $]\inf\{x,y\}, \sup\{x,y\}[ \subset G$ (and analogously with the other intervals).

\begin{definition}\label{def:IVT} Let $\KK $ be a valued subfield of the valued field $\LL$  with the same valuation group $G$ as $\LL$ and let $f: \KK \to \LL$ be a function. Then we say that IVT holds for $f$ if, whenever $v(f(a)) > 0, v(f(b)) < 0$, with $ a,b \in \KK$, there is $c \in \KK$ such that $v(f(c)) = 0$ and $v(c) \in [\{v(a),v(b)\}]$.

We say that IVT holds true for polynomials (power series) if it holds true for every polynomial $f(X) \in \LL[X]$ (power series $S(X) \in \LL[[X]])$.
\end{definition}

Before any investigation on IVT, we want to show that it cannot hold true for every continuous function.

\begin{example}\label{ese:duedue} Let $\KK$ be any valued field  and choose a non-empty subset $U$ such that both $U$ and $\KK\setminus U$ are open (for instance $U$ can be chosen to be the maximal ideal $M$). Let us then choose $k \in \KK$ such that $v(k) = \alpha > 0$.

 Then define the following function $\KK \to \KK$:

$f(x) = k, \forall x \in U$

$f(x) = k^{-1}, \forall x \notin U$.

Obviously $f$ is continuous and IVT fails to hold true on every interval with non-empty intersection both with $U$ and with $\KK\setminus U$.
\end{example}

\subsection{The behaviour of $v(f(x))$}
The following propositions show that, when $f(X) \in \LL[X]$ is a  polynomial and $x$ varies in  $ \KK \subset \LL$, $v(f(x))$ is a function of $v(x)$ outside of a finite subset of $G$ and that such a function can be extended to a continuous function on the whole of $G$.

The piecewise linearity of the function $v(x) \mapsto v(f(x))$ is an elementary fact; nevertheless we include it below because it is a key result in order to establish our intermediate value theorem.

\begin{proposition}\label{prop:duetre} Let $f(X)  \in \LL[X]$ be a  polynomial, with (not necessarily distinct) roots $z_i, i = 1,\cdot \cdot \cdot, n$ in $\bar \LL$. Assume that $ v(z_1) = h_1 \le v(z_2) = h_2 \le \cdot \cdot \cdot \le  v(z_{n-m}) = h_{n-m} < \infty, v(z_{n-m+j}) = h_{n-m+j}=\infty, j=1= 1...m $. Then, for  $x \in \KK^{\times}$, $ \phi_f:  v(x) \to v(f(x))$ is a stepwise linear and non-decreasing function defined on the set $V = G\setminus \{h_1,\cdot \cdot \cdot,h_{n-m}\}$. 
\end{proposition}

 \begin{proof} As  $ z_{n-m+1} = ...= z_n =0$, we can write: 
 $$f(X)= 
q x^m (x-z_1)...(x-z_{n-m}) \hbox{ with }
q \in \LL.$$
 Let $\lambda= v (q)$. Then it holds:
\begin{enumerate}
\item  $ v(f(x)) =\lambda + nv(x)$ if $v(x) < h_1$,
\item $ v(f(x)) = \lambda + h_1+\ldots + h_i +(n-i) v(x) $ if $ h_i < v(x) < h_{i+1}, i\leq n-m-1$,
\item  $ v(f(x)) =\lambda+ h_1+\ldots +h_{n-m}+mv(x)$ if $ v(x) > h_{n-m}$.
\end{enumerate}
Therefore it is obvious that $\phi_f$ is linear and non-decreasing in each set $]-\infty,h_1[, ]h_i,h_{i+1}[, i\leq n-m-1, ]h_{n-m},+\infty[$. Let us now assume that $h_i < v(x) < h_{i+1}$ and  $h_{i+1} < v(y) < h_{i+2}, i+1 \leq n-m$. Then  $ v(f(x))$ satisfies 2, while  $ v(f(y)) = \lambda+h_1+\ldots + h_{i+1} +(n-i-1) v(y) $; it follows that $ v(f(x)) < \lambda+ h_1+\ldots + h_{i+1} +(n-i-1) v(x) < \lambda+h_1+\cdot \cdot \cdot h_{i+1} +(n-i-1) v(y)= v(f(y))$.
Monotonicity obviously holds when $v(x), v(y)$ belong to non-contiguous intervals.

If either $v(x) < h_1$ or $v(y) > h_{n-m}$, the above arguments work and give the proof.
\end{proof}

\begin{proposition}\label{prop:duequattro}  With the notation of Proposition \ref{prop:duetre}, let $z_i,\cdot \cdot \cdot, z_{i+r-1}$ be all the roots whose valuation is $h_i$  (for some $i \le n-m$). Let $x \in \KK$ be such that $ v(x) = h_i$. Then it holds: 
$$v(f(x))  \ge \lambda+h_1+\cdot \cdot \cdot +h_{i-1}+(n-i+1)h_i.$$
Moreover, if $\frac{A}{M}$ is an infinite field,  then there is $x \in \KK$ such that $v(f(x))$ reaches the minimum value $\lambda+h_1+\cdot \cdot \cdot +h_{i-1}+(n-i+1)h_i${\color{red}.}
\end{proposition}

 \begin{proof}
The first claim follows from the inequalities \quad  $ \forall j,  
v(x-z_j) \ge v(x) = h_i, i \le j \le i+r-1$; in fact
$v(f(x))  \ge \lambda+ h_1+\cdot \cdot \cdot+h_{i-1} + rv(x)+(n-r-i+1)h_i   = h_1+\cdot \cdot \cdot h_{i-1}+(n-i+1)h_i$. \\
To prove the second claim  choose $z \in \KK$ such that $v(z) = h_i$ and set: $x = uz, z_j = u_j z$ for $  j = i, \cdot \cdot \cdot, i+r-1$, where $u,u_j$ are all invertible elements in $\overline {\LL}$. Then 
$$v(x-z_j) = v(z)+v(u-u_j) = h_i + v(u-u_j) \geq h_i$$
and the equality is reached if we  
 choose  $u$ such that $\bar u \ne \bar u_j, j = i, \cdot \cdot \cdot,i+r-1$ in $\frac{\overline{A}}{\overline{M}}$ $(\frac{A}{M}$ is required to be an infinite set).

 \end{proof}

\begin{remark}\label{remramification}
$(i)$  As a consequence of the above proposition, $\phi_f$ can be extended to a continuous function on the whole of $G$ by choosing  the minimum value $v(f(x)) = \lambda + h_1+\cdot \cdot \cdot +h_{i-1}+(n-i+1)h_i$ whenever $v(x) = h_i$. The extended function will also be called $\phi_f$. 
Let us point out the obvious relation $\phi_f (x) \le v(f(x))$.\\
\noindent $(ii)$
 The points $v(x)$ where the relation $(v(x), v(f(x)))$ is multi-valued are exactly the values $v(z_i)$, where $z_i$ is a zero of $f(x)$ different from the zero element of $K$.   
\end{remark}

  \begin{lemma}\label{lem:duesei} The notation being as in Propositions \ref{prop:duetre} and \ref{prop:duequattro}, let $h$ be such that $ h_i =\cdot\cdot\cdot = h_{i+r-1} = h$.

Assume that there are a  $w_0 \in \KK$ such that $v(w_0)=h$ and an element $\delta > \phi_f(h), \delta \in G$ with the following property:  $v(f(w_0)) = \delta$. Then, if $\frac{A}{M}$ is infinite and $G$ is divisible, $\forall \gamma \in G, \phi_f(h) \le \gamma < \delta$, there is an  element  $w \in \KK$ such that $v(w)=h$ and $v(f(w)) = \gamma$. 

If moreover $\KK$ is algebraically closed, then, for every $\gamma \ge 0$,   there is  an  element  $w \in \KK$ such that $ v(w) = h ,\  v(f(w)) = \gamma$. 

\end{lemma}

\begin{proof} It is enough to prove  the claim under the hypothesis   
$$v(z_1) = \cdot\cdot\cdot = v(z_n) = h.\quad\quad (*)$$
Indeed let us consider the polynomial 
$$g(X) =q (X-z_i) ( X-z_{i+1}) \cdot\cdot\cdot (X-z_{i+r-1}),$$
satisfying  condition $(*)$. For every $x$ with $v(x) = h $, we have:
$$v(f(x)) = mh+ v(q) + v(g(x))+ \sum _{j<i} h_j + (n-i+r-1) h = v(g(x)) + \eta,$$
where $\eta$ does not depend on $x \in  \{ t:  v(t) = h\}$.
Analogously, 
$$\phi_f(x) = \phi_g(x) +\eta,$$ whenever $v(x) = h$ .    

Hence $v(f(w_0)) =\delta > \phi_f (h)$ is equivalent to  $v(g(w_0)) = \delta -\eta >\phi_g(h)$  and  $\phi_f (w) = \alpha $ is equivalent to $\phi_g(w) = \alpha -\eta$. As a consequence, it is enough to prove the statement for $g(X)$, with $\delta$ and $\alpha$ replaced respectively by $\delta - \eta, \alpha-\eta.$

So, we add condition $(*)$ to the hypothesis and set

$ \lambda = v(q), \  \mu = \phi_f(h) = \lambda + nh$.

Now we notice that it suffices to prove the theorem when $f(X)$ is monic and $h=0$. In fact, suppose the claim true in this case. Given $f(X), \delta,\gamma$ as above, define $e(X) = 
(q w_0^n)^{-1}f(w_0X), \delta'=\delta-\mu, \gamma'=\gamma-\mu$. Then $e(X)$ is monic with unitary roots $u_i = w_0^{-1} z_i$
, $ v(g(1))=\delta '$ and $\gamma'>0$, so that by the theorem there exists $z\in\KK$ such that $v(z)=0$ and $ v(e(z))=\gamma'$. Then put $w=w_0z$.\\
Therefore we assume $
q=1, h=0$  and, as a consequence, $\mu =0.$ 

{\it{Step}} 1. If $\gamma = 0$, thanks to Proposition 2.4 it is enough to select any $w \in \KK$ such that $\bar w \ne \bar u_i \in \frac A M, \forall i$ (the residue field has infinitely many elements). In this case we do not need the hypothesis on the existence of $\delta$ and $w_0$.

{\it{Step 2}}. If $\gamma > 0$,  set: $m_0 = 0, m_i = v(w_0-u_i), \forall i$, where $0 = m_0 \le m_1 \le m_2 \le \cdot \cdot \cdot \le m_n$.

Since $\gamma < \delta = m_1+\cdot \cdot \cdot +m_n$ the following cases can occur.

Case 1. There is an integer $s, 0 \le s \le n-1$ such that 

$m_0+\cdot \cdot \cdot +m_s +(n-s)m_s< \gamma < m_0+\cdot \cdot \cdot +
 m_s+ (n-s)m_{s+1}$, \  $ m_s \not = m_{s+1}$.

(Observe that  $m_s = m_{s+1}$ implies that the two sums coincide and the condition cannot be satisfied).

In this event choose any $x \in \KK$ such that $(n-s)v(x) = \gamma - m_1 - \cdot \cdot \cdot - m_s$; since it holds $m_s < v(x) < m_{s+1}$, such an element $x$ gives rise to $w = w_0+x$ with the required property.

We observe that the open intervals $]m_0+\cdot \cdot \cdot +m_s +(n-s)m_s, m_0+\cdot \cdot \cdot +m_s+ (n-s)m_{s+1}[$ cover the whole $ ]m_0,m_0+\cdot \cdot \cdot+ m_n[$ with the exception of the intermediate endpoints, as a consequence of  the following equalities: 

$m_0+\cdot \cdot \cdot +m_{s+1} +(n-s-1)m_{s+1} =  m_0+\cdot \cdot \cdot +m_s+ (n-s)m_{s+1}$.

Case 2.   
 $(\exists s , s<n , m_s \not= m_{s+1})  \  m_0+\cdot \cdot \cdot +m_s +(n-s)m_s = \gamma$.

   Let us assume that  $m_s =m _{s-1 }= \cdot \cdot \cdot = m_{s-r} > m_{s-r-1}$,

  In this event it is enough to produce an element $x \in K$ with the properties

$$v(x) = m_s, \quad v(x+w_0) = 0,   v (x+(w_0-u_i)) = m_s, i = s-r, \cdot \cdot \cdot, s$$
and choose $w =x+w_0$.

 In fact  we have:
$f(w) = \prod ^n_{j=1} (x+ w_0 -u_j)$  and, as a consequence 

$ v(f(w)) = \sum ^n_{j=1} v(x+(w_0-u_j)) =  m_1+...+m_{s-r-1} +(n-s-r-1)m_s = \gamma$

The element $x\in K$ can be obtained as follows.

 Choose any $y\in K$, such that $v(y) = m_s$ and  set  $(w_0-u_j)= -y t_j, j= s-r,...,s,\   t_j \in \bar K, v(t_j) = 0$.
Then take  $t\in K, v(t) = 0$, such that  $\bar t \not= \bar t_j, j= s-r,...,s$, as elements of $\frac A M$  and set $x=ty$. Then: $v(x) = v(y) = m_s, \  v(x+w_0-u_j) = v(y) +v(t-t_j) =v(y) = m_s $ and $v(x+w_0) = v(w_0) =0$.

 Finally, let us suppose that $\KK$ is  algebraically closed. The proof when $\gamma = 0$ is the same as in Step 1, so we assume $\gamma > 0$.   We want to find an element $w$ such that $v(f(w)) = \gamma \ge 0$, i.e. $v((w-u_1)\cdot \cdot \cdot (w-u_n)) = \gamma > 0$. Select $y \in \KK$ such that $\gamma = v(y)$, so that $v((w-u_1)\cdot \cdot \cdot (w-u_n)) = v(y)$. So, it is enough to choose an element  $w$ satisfying a relation of the form   $w(w^{n-1}+b_{n-2}w^{n-2}+\cdot \cdot \cdot +b_1)= uy-(-1)^nu_1 \cdot \cdot \cdot u_n$, where, $b_1, \cdot \cdot \cdot  b_{n-2}$ are elements with non-negative valuations and $u$ is any invertible. Since $v(uy-(-1)^nu_1 \cdot \cdot \cdot u_n) = 0$, also $v(w) = 0$, for every $w$ solution of the equation in the algebraically closed field $\KK$, which proofs our claim. 

\end{proof}

\begin{remark} 
Observe that the case $\gamma = 0$ requires the only hypothesis that the residue field be an infinite set.

 Let us point out that  this condition is obviously fufilled if $\frac{A}{M}$ has characteristic $0$. It is also fulfilled when $\KK$ is any algebraically closed field,  because in this event  the residue field is still algebraically closed.  
 
  In fact, if $\KK$ is algebraically closed, let $\bar P(X)$ be any monic  polynomial over $\frac{A}{M}$ with degree $ m \ge 2$. Then there is a monic polynomial $P(X) \in A[X]$ such that $\bar P(X)$ is its image mod $M$. Since $\KK$ is algebraically closed, $P(X) = (X-x)Q(X)$, where $x \in A$, since $x$ is integral over the integrally closed ring $A$. Therefore $\bar P(\bar x) = 0$, so that $\bar P(X)$ has a root in $\frac A M$, i.e. the residue field is algebraically closed.

\end{remark}
We can summarize the above results concerning  the behaviour of $v(x)\mapsto v(f(x))$ in the following: 
\begin{theorem}\label{teo:polynomialbehaviour} Let $\KK$ be a valued field satisfying conditions (\ref{eq:Hypo}) and extend the valuation $v$ to any valuation of  $\overline{\KK}$.
Let $f(X)\in \KK[X]$ be a polynomial. Set   $T=\{x\in \overline{\KK}\ |\ f(x)=0\}$. Then 
\begin{itemize}
\item[a)] $\phi_f:v(x)\mapsto v(f(x))$ is a stepwise linear and non-decreasing function defined on the set $V=G\setminus v(T)$.
\item[b)] Let $\alpha\in G$. Then the set $\{v(f(x))\ |\ x\in\KK \hbox{ and } v(x)=\alpha\}$ is a left-bounded and left-closed interval in $G$.
\item[c)] The map $\phi_f$ can be extended to a continuous function on the whole of $G$ by choosing for $\alpha\in G$
$$\phi_f(\alpha)=\min\{v(f(x))\ |\ x\in\KK\hbox{ and } v(x)=\alpha \}.$$ 
\end{itemize}
\end{theorem}

We are now ready to state and prove the intermediate value theorem for polynomials.

\begin{theorem}[Polynomial intermediate value theorem] \label{prop:dueotto} Let $f(X) \in \LL[X]$ be any  polynomial such that $ v(f(a)) > v(f(b))$ for some $a,b \in \KK$, where $v(a),v(b)$ are not necessarily distinct. Then for every $\alpha \in G$ such that $ v(f(a)) > \alpha >  v(f(b))$ there is $c \in \KK$ such that:
\begin{enumerate}
\item $ v(f(c)) = \alpha$,
\item  $v(c) \in [\{v(a),v(b)\}]$.
\end{enumerate}
\end{theorem}

\begin{proof}

 We can suppose that $f(X)$ is monic.\\
  
  Since  $\phi_f(v(b))\leq  v(f(b))$, we have $\phi_f(v(b))<\alpha$. We consider the following cases.

Case 1.  $\alpha \leq \phi_f(v(a))$.  Since $\phi_f$ is non-decreasing  and    $\phi_f(v(b))<\alpha \leq \phi_f(v(a))$, we have necessarily $v(b)\leq v(a)$.  As $\phi_f$ is stepwise linear and  continous on $[\{v(a),v(b)\}] = [v(b),v(a)]$, there is an element $\eta \in [v(b),v(a)]$ such that $\phi _f(\eta) = \alpha$.
Indeed, choose a sub-interval  $[v(b'),v(a')] \subset [v(b),v(a)]$ such that $\phi_f$ is linear on  it,  say $\phi_f(v(x)) = A+B v(x), B \ne 0$,  and $\phi_f(v(b')) \leq \alpha \leq \phi_f(v(a'))$. Then $\eta = \frac{\alpha-A}{B}$ works,
So, every $c$ such that $v(c) = \eta$ is a solution.

Case 2. $\phi_f(v(a)) < \alpha $.   This implies  $\phi_f(v(a)) < v(f(a))$; hence $v(a)$ coincides with the valuation of some roots of $f(X)$, so that we can conclude by using Lemma \ref{lem:duesei}.

Observe that in this case $v(c) = v(a)$.

\end{proof}

 \begin{remark} \label{infinity} Notice that the proof of Theorem \ref{prop:dueotto} (together with the preliminary results) actually shows that  the elements $c$ satisfying conditions $1$ and $2$  are infinitely many.
 \end{remark}

The following proposition shows that our two conditions on $\KK$ are necessary.

\begin{proposition}\label{prop:duenove} Let $\KK$ be a valued field such that IVT holds true for every polynomial $P(X) \in \LL[X]$. Then conditions (\ref {eq:Hypo}) are fulfilled.

\end{proposition}

\begin{proof}
(i) Assume that $\frac{A}{M}$ is a finite field of order $q$. Let $\{z_1, \cdot \cdot \cdot, z_q\}$ be a set of representatives of the elements in $\frac{A}{M}$, and set:

$$P(X)= (X-z_1)(X-z_2)...(X-z_q).$$ 

Let $x$ be in $\KK$; then 

a) if $v(x) < 0$, then $v(P(x)) = qv(x) < 0$,

b) if $v(x) > 0$, then $v(P(x) = 0$,

c) if $v(x) = 0$, then $v(P(x)) > 0$. 

Properties a) and b) being obvious, let us consider an element $x$ such that $v(x) = 0$; since $x \in A$ we have: $\bar x = \bar z_i$ for exactly one element $z_i$.  Therefore $v(x-z_i) > 0 $ and $v(x-z_j) = 0$ whenever $j \ne i$, so that $v(P(x)) = \sum_jv(x-z_j) > 0$. 

 Let us now choose any interval $[a,x]$ such that $ v(a) < 0, v(x) = 0$. Then $v(P(a)) < 0, v(P(x)) > 0$ and $v(a) \le v(c) \le v(x)$ implies $v(P(c)) \ne 0$, so that IVT does not hold true.
  
  (ii) Assume that there are $a \in \KK, h \in G, n \in \NN, n \ge 1$ with the following property: $v(a) = h$ and there is no $x \in G$ such that $nx = h$. By replacing, if necessary,  $a$ with $a^{-1}$, we can assume that $h > 0$.

  Set: $P(X) = X^n$.  Let us choose $a \in \KK$ such that $nv(a) > h$ (this is obviously possible in every ordered group). Notice that $v(P(a^{-1})) = -nv(a) < 0 < h < v(P(a)) = nv(a)$. Assume now that IVT holds true  for $P(X)$; this imples that there is $c \in \KK$ such that $nv(c) = h$.  This is a contradiction.

\end{proof}

  \bigskip

\section{Power series} 

We fix once for all a field $\KK$ equipped with a valuation $v$ with divisible group $G$, having $(A,M)$ as its valuation ring with infinite residue field. We also fix an overfield $\LL \supset \KK$, equipped with a valuation extending $v$ and whose group is $G$; we will call $v$ also this valuation, since there is no ambiguity.  Let $\tau$ be the  topology induced by $v$ on the field $\LL$. 

Since we study countable converging sequences, we are forced to  assume that $\tau$ has a countable basis for the set of neighbourhoods of $0$, because otherwise the only converging series are the polynomials.

We also assume from now on that  $\LL$  is complete. 

The valuation ring $(A_L,M_L)$ is separated  and complete with respect to  $\tau$. It is worth observing that the sets $\{x \in A_L, v(x) > \alpha\}$ are, for every $\alpha \in G$, ideals of $A_L$, so that $\tau$ is a linear topology on $A_L$; moreover $M_L$ is easily seen to be a closed ideal. We will call $\tau$ also the restriction of $\tau$ to $\KK$. 

Such a restriction needs not be complete, but is still a separated linear topology.

Despite the countability restriction, our results hold true for a wide class of valued fields, far away from every $p$-adic field, as it is shown in the following remark and example.

We recall that the rank of a valuation is the length of a maximal chain of non-trivial convex subgroups of the valuation group $G$ (see \cite{EP}, section 2.1, p. 26); it coincides with the Krull dimension of the valuation ring (see \cite{EP}, Lemma 2.3.1). It is well-known that  rank $1$ and  real are equivalent (see \cite{EP}, Proposition 2.1.1)

We also recall (see \cite{Ribenboim1}, p. 66) that $\tau$ has a countable basis if and only if either the rank of $G$ is finite or the zero ideal is a countable intersection of prime ideals.

\begin{remark} Let $G$ be a non-trivial abelian ordered group containing a proper maximal convex subgroup (in particular this holds true when $G$ has finite rank). Then there is  $x \in G$ such that $\lim_{n \to \infty} nx = \infty$, i.e. $\forall g \in G, \exists n \in \NN$ such that $nx > g$.

Indeed, let $H$ be a proper convex maximal subgroup and let $x$ be in $G\setminus H$. and $y$ in  $G$. Since the convex subgroup generated by $H$ and $x$ is $G$ itself, it holds: $\exists h \in H, \exists n \in \NN$ such that $ y < h+nx < x + nx = (n+1)x$.

We observe that, if $\KK$ is a valued field whose group satisfies the above conditions, then $\KK$ is forced to contain at least one topologically nilpotent element.
\end{remark}

The following example shows that a countable basis can exist even when $\KK$ contains no topologically nilpotent element.

\begin{example} 

Step 1. Let $A_1 = \QQ[\epsilon_1]$, where $\epsilon_1$ is transcendental over $\QQ$. Then $\KK_1 = \QQ(\epsilon_1)$ is the quotient field of $A$.  
 
 We set: $\bar v(\sum_{i = 0}^n a_i \epsilon_1^i) = $ min $\{i | a_i \ne 0\}$.
 
 Such a function  $\bar v$ can be obviously  extended to a valuation $v: \KK_1 \to \ZZ$.
 
Step 2. Let $A_n = \QQ[\epsilon_1, \cdot \cdot \cdot, \epsilon_n]$, where $\epsilon_1, \cdot \cdot \cdot ,\epsilon_n$ are algebraically independent over  $\QQ$. Let $\KK = \QQ(\epsilon_1, \cdot \cdot \cdot, \epsilon_n)$ be the quotient field of $A$ and let us order lexicographically the group $\ZZ^n$ and put: $\bar v(\epsilon_i) = (0,\cdot \cdot \cdot,1_i,\cdot \cdot \cdot, 0), i = 1, \cdot \cdot \cdot ,n, \ \bar v(a) = 0, a \in \QQ$.
 
 We set: 
 
 $\bar v_n(a_{i_1 \cdot \cdot \cdot i_n} \epsilon_1^{i_1} \cdot \cdot \cdot \epsilon_n^{i_n}) = (i_1,\cdot \cdot \cdot, i_n)$ if  $a_{i_1 \cdot \cdot \cdot i_n} \ne 0$,
 
 $v(\sum_{(i_1 \cdot \cdot \cdot i_n)}a_{i_1 \cdot \cdot \cdot i_n} \epsilon_1^{i_1} \cdot \cdot \cdot \epsilon_n^{i_n}) = $ min $\{(i_1,\cdot \cdot \cdot , i_n),  a_{i_1 \cdot \cdot \cdot i_n} \ne 0\}$.
 
 Such a function, as above, can be extended to a valuation $v_n: \KK \to \ZZ^n$.
 
 Step 3. Set: $A = \cup A_n, \KK = \cup \KK_n, G = \ZZ ^{\NN}$ (with lexicographic order) and define $ v:\KK \to G $ as the function such that $v_|(\KK_n) = v_n, \forall n \in \ZZ$. Then set: $v(0) = \infty$. It is clear that $v$ is a valuation whose group is $G$.
 
 Step 4. $v$ can be extended (in at least one way) to the algebraic closure $\bar \KK$ and the corresponding group is the smallest divisible group $\bar G$ containing $G$, i.e. $\bar G$ can be identified with $\QQ^{\NN}$.
  
 We notice that  $\bar G$ is a non-Archimedean countable ordered group. Therefore the topology of $\KK$, as well as the topology of its completion $\hat{\bar \KK}$ as a valued field{\color{green},} is countable. 
 
 It is worth observing that  this  topology has a countable basis for the neighbourhoods of $0$. Moreover $\KK$ has no topologically nilpotent element.

\end{example}

A series $\sum a_n, a_i \in \LL$ converges if and only if $\lim a_n = 0$. The {\it{only if}} implication is true in general because, $a_n = \sum_0^na_i - \sum_0^{n-1}a_i$. As for the {\it{if}} implication, it is enough to point out that $v(\sum_N^M a_i) \ge $ min $\{v(a_N), \cdot \cdot \cdot, v(a_M)\} \ge \gamma, \forall N \gg 0$.

  If $S(X) = \sum a_nX^n \in \LL[[X]]$ and $x \in \LL$, then $S(x) = \sum a_nx^n$, if existing, is an element of $\LL$. In particular, if $x \in \KK$ and $S(x)$ is converging,  then  $S(x) \in \LL$. 
  
  
   \begin{lemma}\label{lem:Uniforme} If a power series $S(X) = \sum a_nX^n \in \LL[[X]]$  converges at $x$, then it converges  at every $y$ such that $v(x) \le v(y)$ and such a convergence is uniform.
   
 \end{lemma}
  
\begin{proof}Assume that there is $x \in \LL$ such that $S(X)$ is convergent at $x$ to $S(x) \in \LL$. Then it holds : $\lim a_nx^n = 0$, i.e., given $\alpha \in G$, there is $N$ such that $\forall n > N, v(a_nx^n) = v(a_n)+nv(x) > \alpha$. If $ v(x) \le  v(y)$, then $v(a_ny^n) = v(a_n)+nv(y) > \alpha$, so that $S(X)$ is also convergent at $y$.

As for uniform convergence, let $\gamma \in G$ be any element, then there is $N$ such that, $\forall M > N, v(\sum_N^M a_ix^i) \ge \gamma.$ As a consequence, for every $y$ such that $v(y) \ge v(x)$ $$v(\sum_N^M a_iy^i) \ge \min_{N \le i \le M}(v(a_i)+iv(y)) \ge \min_{N \le i \le M}(v(a_i)+iv(x)) \ge \gamma.$$

 \end{proof}

\begin{theorem}\label{teo:IVTuniforme} Let $(f_n(X): D = D_{\alpha,\beta} \to \LL, N \in \NN)$ be a sequence of functions, where $D_{\alpha,\beta} = \{x \in \KK, \alpha \le v(x) \le \beta\}$. Let us assume that the sequence converges uniformly on $D$ to a function $f(X): D \to \LL$.

If  each $f_n(X)$ satisfies IVT on $D$, then the same is true for $f(X)$.

\end{theorem}
\begin{proof}

First of all we notice that if $\lim_{n \to \infty} h_n = h \ne 0, h_n, h \in \KK$, then there is $N \in \NN$ such that $v(h) = v(h_n), \forall n \ge N$. Indeed choose $\delta \in G, \delta > v(h)$. Then there is $N \in \NN$ such that $v(h-h_n) > \delta, \forall n \ge N$. It follows that $v(h) = v(h_n)$.

 Let us choose $a,b \in D$ and assume  that $v(f(a)) < \gamma < v(f(b)) \leq \infty$. 

If $v(f(b)) < \infty$,  we choose  $\delta > v(f(b))$ and observe that, by the hypothesis on the uniform convergence of the sequence, there exists  $n(\delta)$ such that
$$ n>n(\delta) \Longrightarrow  v(f_n(x)) -v(f(x)) > \delta, \forall x \in D_{a,b}.$$
Now, let us fix any $N>n(\delta)$. We have:
$$v(f_N(a)) = v(f(a)), \  v(f_N(b)) = v(f(b));$$
moreover, by the I.V.T. condition on $f_N(x)$, there exists $c_N \in D_{a,b}$ such that  $f_N(c_N) = \gamma$.

As  $v(f_N(c_N) - f(c_N)) >\delta > \gamma$ , we get $v(f(c_N)) = \gamma$.

Analogously, if $v(f(b)) = \infty$, there exists $n(\gamma)$ such that 
$$n> n(\gamma) \Longrightarrow v(f_n(x) - f(x)) > \gamma, \  \forall x \in D_{a,b}.$$ Moreover $\exists m$ such that
$$n > m \Longrightarrow v(f_n(b)) > \gamma > v(f_n(a)).$$

As in the previous case, IVT, applied to any $N> $ max $\{n(\gamma), m\}$, implies that there exists $c_N \in D_{a,b}$ such that  $v(f_N(c_N))=\gamma$  and, as a consequence, $v(f(c_N)) = \gamma$.

So, in both cases, for every  sufficiently large $N$, we have at least one element $c_N$ satisfying IVT  on $D_{a,b}$.

  \end{proof}

From Lemma \ref{lem:Uniforme} and Theorem \ref{teo:IVTuniforme} we deduce IVT in $\KK$ for every power series belonging to $\LL[[X]]$.

\begin{theorem} \label{teo:IVTseries} Let $S(X) = \sum b_nX^n \in \LL[[X]]$ be a power series and let $a,b \in \KK$  such that $v(a) \le v(b)$. We assume that $S(X)$   is convergent at $a$ and  that $S(a) \ne S(b)$. If  $\alpha \in [\{v(S(a)),v(S(b))\}]$, then  there is $c \in \KK$ such that $v(S(c)) = \alpha$ and moreover $v(c)$ belongs to $[v(a),v(b)].$
\end{theorem}

  \begin {remark} \label {generalizzazione} If $S(X)$ takes values in an overfield $\KK'$ of $\KK$, whose valuation extends the valuation  of $\KK$, still satisfying conditions (\ref{eq:Hypo})  but not necessarily complete, we can replace $\KK'$ by its completion $\LL$ and apply Theorem \ref{teo:IVTseries} , choosing $\gamma$ in $\KK' \subset\LL$.  So Theorem \ref{teo:IVTseries} is true also when $\LL$ is not  complete.
  \end{remark}

\section{Hensel's Lemma and applications}
\subsection{IVT for power series}
 In this section we give an alternative proof of Theorem \ref{teo:IVTseries} based on Hensel's Lemma.
 
We recall  (see \cite[\S 3]{Greco-Salmon}) the following 

\begin{definition}  $S(X) = \sum a_nX^n \in A[[X]]$ is \emph{restricted} if $ \lim_{n \to \infty} a_n = 0$. The ring of restricted power series is denoted by $A\{X\}$.

A restricted power series is called \emph{regular} if at least one of its coefficients is a unit.

\end{definition} 

$S(X) = \sum a_nX^n \in A[[X]]$   is  restricted if and only if it is convergent everywhere on $A$; in this event the convergence is uniform (Lemma \ref{lem:Uniforme} with $x = 1$).

In what follows we will use a strong version of Hensel's Lemma which holds true for restricted power series. We start with the following 

\begin{definition} \label{def:henselian}
A local ring $(A,M)$ is called \emph{henselian} if the following property holds:

Let $P(X) \in A[X]$ be a polynomial such that its canonical image $\bar {P}(X)$ into the quotient  ring $\frac{A}{M}[X]$ is the product $\bar{Q}(X) \bar{T}(X)$ of a monic polynomial $\bar{Q}(X)$ and another polynomial $\bar{T}(X)$, the two factors being coprime. Then $P(X) = Q(X)T(X)$, where $Q(X)$ is a monic polynomial that lifts $\bar{Q}(X)$ and $T(X)$ is a polynomial that lifts $\bar{T}(X)$. Moreover $P(X),Q(X)$ are uniquely determined and coprime.

\end{definition}

Hensel's Lemma states that a local ring $(A,M)$ that is complete with respect to the $M$-topology  is henselian (Nagata).  

It is worth to point out that, if $(A,M)$ is a valuation ring, the valuation topology differs from the $M$-adic topology unless the rank is $1$. In general the completion with respect to the valuation topology does not satisfy Hensel's Lemma (see \cite{EP}, p. 50). However, every valued field which is "stage complete"  is simultaneously complete and henselian (see \cite{Ribenboim1} p. 74 ex. 6 and p. 198 Th. 4).  Moreover every $\KK$ admits a \lq\lq stage complete \rq\rq extension  (see \cite[p. 88, Th. 1,  p. 112, Cor. 2 ]{Ribenboim1}).

In this section we assume that $\LL \supset \KK$ has a  topology with a countable basis and that it is henselian and complete (such properties can be attained by considering the "stage completion" of $\KK$). 

Hensel's Lemma can also be given for restricted power series.

\begin{lemma} \label 
{lem:HenselRestrictedSeries} (Hensel's Lemma for restricted power series) Let $A$ be a complete separated ring with respect to a linear topology, $M$ a closed ideal. Assume that Hensel's Lemma for polynomials holds true in $A$. Let $S(X)$ be a restricted power series such that its canonical image $\bar {S}(X)$ into the topological quotient  ring $\frac{A}{M}\{X\}$ is the product $\bar{P}(X) \bar{T}(X)$ of a monic polynomial $\bar{P}(X)$ and a restricted series $\bar{T}(X)$, the two factors being coprime. Then $S(X) = P(X)T(X)$, where $P(X)$ is a monic polynomial that lifts $\bar{P}(X)$ and $T(X)$ is a restricted series that lifts $\bar{T}(X)$. Moreover $P(X),T(X)$ are uniquely determined and coprime.
\end{lemma}

\begin{proof} With topologically nilpotent elements of $M$, it is \cite{Bourbaki}, \S 4, p. 84, Th\' eor\`eme 1. Without topologically nilpotent elements it is \cite{Valabrega}, Teorema 5.
\end{proof}

\begin{corollary} \label{cor:Salmon} Let $\LL$ be henselian and complete. Let $S(X) = \sum_{n = 0}^{\infty} a_nX^n$ be a restricted power series over $A$ such that the partial sum $S_N(X)$ is a monic polynomial for some $N$ and moreover $a_{N+h} \in M, \forall h \ge 1$. Then $S(X) = P(X)B(X)$, where $P(X)$ is a monic polynomial such that $P(X) = S_N(X)$ mod $M$ and $B(X) \in 1+M\{X\}$ is a restricted power series.
\end{corollary}

\begin{proof} This is essentially \cite[ Th\' eor\`eme 10]{Salmon}. In fact the proof of this theorem makes only use of Hensel's Lemma for restricted power series, applied to the decomposition (mod $M$): $\bar S(X) = \bar S_N(X) \cdot \bar 1$. Salmon's proof works with topologically nilpotent elements, but by using \cite{Valabrega}, Teorema 5, this condition can be avoided. 

 \end{proof}

\bigskip

\bigskip

Let now $\KK$ be any valued field satisfying conditions $(i)$ and $(ii)$ of Proposition \ref{prop:duenove} and let $\LL$ be any complete henselian extension of $\KK$, having the same value group as $\KK$. Then Hensel's Lemma (Lemma \ref{lem:HenselRestrictedSeries})
 holds true for any restricted power series $S(X) \in A_L\{X\}$.

In what follows we will need, starting with a point $a \in \LL, a \ne 0$,  to define a trasformation of a power series converging at $a$ into a restricted regular one. Therefore we introduce the following 

\noindent{\it{Construction.}}

Let $S(X)\in \LL[X]$ be a series converging at $a\in \LL, a \ne 0$. We define \begin{equation}\label{eq:construction} S_a(X)=h_aS(aX)\in A_\LL\{X\}\end{equation} where $h_a\not=0$ is such that $\overline{S_a(X)}$ is a monic polynomial, i.e. $S_a(X) $ is a regular restricted series.\\
  Since $S(aX)$ converges at $X = 1$, the coefficients whose valuation is less than or equal to a fixed value form a finite set.  Then the element $h_a$ can be chosen observing that $h_a^{-1}$ must be any coefficient of $S(aX)$ having lowest valuation.

By Corollary \ref{cor:Salmon}  $$S_a(X)=P_a(X)B_a(X)$$ where $P_a(X)$ is a polynomial and $B_a(X)\in A_\LL\{X\}$ is such that $\overline{B_a}(X)=\mathbf{1}$.

If $x\in A_\LL$ then $S_a(x)$ converges and $v(S_a(x))=v(P_a(x))$. Therefore
if $v(a')\geq v(a)$
\begin{equation}
\label{eq:Sa}
v(S(a'))=v(h_a^{-1}S_a\left(\frac{a'} a\right))=v(P_a\left(\frac{a'} a\right ))-v(h_a).\end{equation}

We are now ready for an 

\noindent{\it{Alternative proof of Theorem \ref{teo:IVTseries}}}

We use the above construction, consider $S_a(X)$ and observe that our theorem is proved if we show that $\exists c' = \frac{c}{a} \in \KK, 0 \le v(c') \le v(\frac{b}{a})$ such that  $v(h_a)+\alpha = v(S_a(c'))$. Now, by Hensel's Lemma $S_a(X) = P_a(X)B_a(X)$, where $v(B_a(x)) = 0, \forall x \in A_L$. This implies that IVT for the power series is equivalent to IVT for the polynomial $P_a(X)$, which has been proved in section \ref{sect:due} (Theorem \ref{prop:dueotto}).

\subsection{The behaviour of $v(x)\mapsto v(S(x))$ for power series}
Let $\KK$ be a valued field satisfying conditions (\ref{eq:Hypo})
 with a countable basis.
In this section, we extend Theorem \ref{teo:polynomialbehaviour}  to the power series case (see Theorem \ref{teo:seriesbehaviour} below).

Let $\LL$ be a complete henselian extension of $\KK$, with the same value group $G$.
Let $S(X)\in \LL[[X]]$ be a series. 
Let $\KK'$ be any valued extension of $\LL$  with value group $G'\supseteq G$; for any $\alpha\in G'$ define:
\begin{itemize}
\item $C_{\alpha,\KK'} =\{x\in\KK' \ |\ v(x)=\alpha \}$
\item $D_{S,\KK'}=\{x\in\KK' \ |\ S(X) \hbox{ converges at }  x\}$
\item $Z_{S,\KK'}=\{x\in D_{S,\KK'} |\ S(x)=0 \}$
\end{itemize}

\begin{proposition} Suppose that $\KK'$ satisfies the following condition:
\begin{equation}\label{eq:condiz} \hbox{ if $x\in D_{S,\KK'}$, then there exists $a\in D_{S,\LL}$ such that $v(a)\leq v(x)$.} \end{equation}
Then every element in $Z_{S,\KK'}$ is algebraic over $\LL$.

\end{proposition}
\begin{proof}
Let $x\in Z_{S,\KK'}$ and let $a\in\LL$ as in condition (\ref{eq:condiz}).
We can apply the construction $(\ref{eq:construction})$ and write
$$S(aX)=\frac 1 {h_a}P_a(X)B_a(X)$$
 where $h_a\in\LL$, $P_a(X)$ is a polynomial in $\LL[X]$ and $B_a(X)\in A_\LL\{X\}$ is such that $\overline{B_a}(X)=\mathbf{1}$. Put $y=\frac x a\in A_{\KK'}$. Since $S(ay)=0$ then $P_a(y)=0$ so that $y$ (and therefore $x$) is algebraic over $\LL$.
\end{proof}

Notice that condition \ref{eq:condiz} is obviously satisfied when $G'=G$.

\begin{theorem}\label{teo:seriesbehaviour} The notation being as above, 
\begin{itemize}
\item[a)] Let $\alpha\in G$. If $S$ converges on $C_{\alpha,\KK}$ and there exist $a,b\in C_{\alpha,\KK}$ such that $v(S(a))\not= v(S(b))$ then  there exists $c\in \overline{\LL}$ such that $v(c)=\alpha$ and $S(c)=0$. Moreover for every $\gamma\in G$ lying between $v(S(a))$ and $v(S(b))$ there is an element $c'\in C_{\alpha,\KK}$ such that $v(S(c'))=\gamma$.
\item[b)] Suppose that $S$ is not the zero series and define the set 
$$H_S=\{\alpha\in G\ |\ \exists x\in \overline{\LL} \hbox{ such that } S(x)=0, v(x) = \alpha\}.$$ For every $\alpha \in H_S$ the set
$$\{x\in Z_{S,\overline \LL}\ | v(x) \geq \alpha\}$$
is a finite set. 

As a consequence the set
$$\{\beta\in H_S\ |\ \beta >\alpha\}$$
is a finite set.
\item[c)] $\phi_{S,\KK}:v(x)\mapsto v(S(x))$ is a stepwise linear and non-decreasing function defined on the set $V=v(D_{S,\KK})\setminus H_S$.
\item[d)] Let $\alpha\in v(D_{S,K})$. Then the set $\{v(S(x))\ |\ x\in C_{\alpha,\KK}\}$ is a left-bounded and left-closed interval in $G$.
\item[e)] The map $\phi_{S,\KK}$ can be extended to a continuous function on the whole of $v(D_{S,\KK})$ by choosing for $\alpha\in v(D_{S,\KK})$
$$\phi_{S,\KK}(\alpha)=\min\{v(S(x))\ |\ x\in C_{\alpha,\KK} \}.$$ 
\end{itemize}
\end{theorem}
\begin{proof}
(a) By formula \ref{eq:Sa} we have $v(P_a(\frac b a))\not =v(P_a(1))$ so that there exists $c'\in C_0$ such that $P_a(c')=0$; put $c=c'a$; then $c\in C_\alpha$ and $S(c)=0$. The second assertion is a consequence of Theorem \ref{teo:IVTseries}.\\
(b) Let $\alpha\in H_S$; let  $a\in\overline \LL$ such that $v(a)=\alpha$ and $S(a)=0$.  Let $b\in \overline\LL$ such that $S(b)=0$ and $v(b)=\beta \geq \alpha$. Then $\frac b a\in A_{\overline{\LL}}$ must be a zero  of the polynomial $P_a(X)$ in construction (\ref{eq:construction}) and there are only finitely many such zeros. \\
(c)-(d)-(e) Let $\beta\in v(D_{S,\KK})$ and put $D_{S,\KK,\beta}=D_{S,\KK}\cap\{ x\in \KK\ |\ v(x)\geq \beta\}$. Since $\beta$ is arbitrary, it suffices to show that the claims hold if we replace $D_{S,\KK}$ by $D_{S,\KK,\beta}$. Let $b\in\KK$ be such that $v(b)=\beta$;  if $v(a)\geq \beta$, formula (\ref{eq:Sa}) shows that $v(S(a))$ is simply a translation of $v(P_b(\frac {a}{b}))$  and the result follows from the polynomial case (Theorem \ref{teo:polynomialbehaviour}).

\end{proof}

\begin {remark} \label {generalizzazione 2} Just as happens in Remark\ref{generalizzazione}, $S(X)$ can take its values in any  
  overfield $\KK'$ of $\KK$, whose valuation extends the valuation  of $\KK$, which is still satisfying conditions (\ref{eq:Hypo})  and in this situation the field $\LL$ of Theorem \ref{teo:seriesbehaviour} is any overfield of $\KK'$, that is at the same time henselian and complete. 
\end{remark}

Authors' addresses:\\

Carla Massaza  

Dipartimento di Scienze Matematiche 

Politecnico di Torino 

Corso Duca degli Abruzzi 24 

10129 Torino Italy

email: carla.massaza@polito.it

\bigskip

Lea Terracini

Dipartimento di Matematica

Universit\`a di Torino

via Carlo Alberto 10 

10123 Torino Italy

email: lea.terracini@unito.it

\bigskip

Paolo Valabrega

Dipartimento di Scienze Matematiche 

Politecnico di Torino 

Corso Duca degli Abruzzi 24 

10129 Torino Italy

email: paolo.valabrega@polito.it

\bigskip

Aknowledgements

The paper was written while P. Valabrega was members of Indam-Gnsaga and L. Terracini and P. Valabrega were supported by the Ministry grant 2010-2011 MIUR-PRIN Geometria delle Variet\`a Algebriche.

\bigskip

\end{document}